\documentclass[a4paper,12pt]{amsart}

\usepackage[left=3.0cm, right=3.0cm, top=2.50cm, bottom=2.50cm]{geometry}

\usepackage{amsmath}
\usepackage{amssymb} 
\usepackage{amsfonts}
\usepackage{enumerate}
\usepackage{graphicx}

\newcommand{\UP}{\blacktriangle}
\newcommand{\DOWN}{\blacktriangledown}

\theoremstyle{plain}
\newtheorem{theorem}{Theorem}[section]
\newtheorem{proposition}[theorem]{Proposition}
\newtheorem{lemma}[theorem]{Lemma}
\newtheorem{corollary}[theorem]{Corollary}

\theoremstyle{definition}
\newtheorem{definition}[theorem]{Definition}
\newtheorem{example}[theorem]{Example}
\newtheorem{remark}[theorem]{Remark}

\begin{document}
 
\title[Prime filter structures of pseudocomplemented Kleene algebras]{Prime filter structures of pseudocomplemented 
Kleene algebras and representation by rough sets}

\author{Jouni J{\"a}rvinen}
\address[J.~J{\"a}rvinen]{Department of Mathematics and Statistics, University of Turku, 20014 Turku, Finland}
\email{jjarvine@utu.fi}

\author{S{\'a}ndor Radeleczki}
\address[S.~Radeleczki]{Institute of Mathematics, University of Miskolc, 3515~Miskolc-Egyetemv{\'a}ros, Hungary}
\email{matradi@uni-miskolc.hu} 
           
\thanks{The work of the second author was carried out as a part of the EFOP-3.6.1-16-00011 ``Younger and
Renewing University -- Innovative Knowledge City'' project implemented in the framework of the
Sz{\'e}chenyi 2020 program, supported by the European Union, co-financed by the European Social
Fund.}
           
\date{}

\subjclass[2010]{Primary 06B15; Secondary 06D15, 06D30, 68T37}

\keywords{Pseudocomplemented Kleene algebra, regular double $p$-algebra, Stone identity,
prime filter, rough set, tolerance induced by an irredundant covering}

\begin{abstract}
We introduce Kleene--Varlet spaces as partially ordered sets equipped with a polarity satisfying certain additional conditions. 
By applying Kleene--Varlet spaces, we prove that each regular pseudocomplemented Kleene algebra is isomorphic to a subalgebra of the rough set
regular pseudocomplemented Kleene algebra defined by a tolerance induced by an irredundant covering. We also characterize the Kleene--Varlet 
spaces corresponding to the regular pseudocomplemented Kleene algebras satisfying the Stone identity.
\end{abstract}

\maketitle
\section{Introduction to rough set algebras} \label{Sec:Intro}

Rough set theory was introduced by Z.~Pawlak in \cite{Pawl82}. His idea was to develop a formalism for dealing with vague concepts and sets. In rough set
theory it is assumed that our knowledge is restricted by an indistinguishability relation. Originally Pawlak defined an \emph{indistinguishability relation} 
as an equivalence $E$ such that two elements of a universe of discourse $U$ are $E$-equivalent if we cannot distinguish these two elements by their properties known by us. 
For instance, we may know some properties of human beings $U$ such as age, gender, height, and weight. Then $x \, E \, y$ means that the persons 
$x$ and $y$ are indistinguishable by these properties.

Each subset $X$ of $U$ can now be ``approximated'' by using the indistinguishability relation $E$. Let us denote the equivalence class
of $x$ by $[x]_E$, that is, $[x]_E = \{ y \in U \mid x \, E \, y\}$. The \emph{lower approximation}
\[X^\DOWN = \{ x \in U \mid {[x]_E} \subseteq X \} \]
of $X$ may be viewed as the set of elements belonging certainly to $X$ in view of the knowledge $E$, because if $x \in X^\DOWN$, then all elements 
indistinguishable from $X$ are in $X$. The \emph{upper approximation}
\[X^\UP = \{ x \in U \mid { [x]_E } \cap X \neq \emptyset \} \]
of $X$ can be seen as the set of elements belonging possible to $X$ by the means of the knowledge $E$. Indeed, if $x \in X^\UP$, then
$X$ contains at least one element indistinguishable from $x$.

Let us denote by $\wp(U)$ the \emph{powerset} of $U$, that is, $\wp(U) = \{ X \mid X \subseteq U\}$. We define a relation $\equiv$ on $\wp(U)$ by
\[ X \equiv Y \iff X^\DOWN = Y^\DOWN \text{ \ and \ } X^\UP = Y^\UP. \]
The relation $\equiv$ is an equivalence called \emph{rough equality}. If $X \equiv Y$, then the same elements belong certainly and
possibly to $X$ and $Y$ in view of the knowledge $E$. The equivalence classes of $E$ are called \emph{rough sets}.

The order-theoretical study of rough sets was initiated by T.~B.~Iwi{\'n}ski in \cite{Iwin87}. 
In his approach rough sets on $U$ are the pairs $(X^\DOWN,X^\UP)$, where $X \subseteq U$.
This is justified because if $\mathcal{C} \subseteq \wp(U)$ is a rough set as defined before, that is,
$\mathcal{C}$ is an equivalence class of $\equiv$, then $\mathcal{C}$
is uniquely determined by the pair $(X^\DOWN,X^\UP)$, where $X$ is any member of $\mathcal{C}$:
a set $Y \subseteq U$ belongs to $\mathcal{C}$ if and only if $(Y^\DOWN, Y^\UP) = (X^\DOWN,X^\UP)$.
Therefore, we call
\begin{equation*}
\mathit{RS} = \{ (X^\DOWN,X^\UP) \mid X \subseteq U \}
\end{equation*}
the \emph{set of rough sets}. The set $\mathit{RS}$ is ordered by the componentwise inclusion: 
\begin{equation*}
(X^\DOWN,X^\UP) \leq (Y^\DOWN,Y^\UP) \iff X^\DOWN \subseteq Y^\DOWN \mbox{ \ and \ } X^\UP \subseteq Y^\UP. 
\end{equation*}
Let $X^c = U \setminus X$ denote the set-theoretic \emph{complement} of $X \subseteq U$. Iwi{\'n}ski noted that the map 
${\sim} \colon \mathit{RS} \to \mathit{RS}$ defined by ${\sim} (X^\DOWN, X^\UP) = (X^{c \DOWN}, X^{c \UP}) = (X^{\UP c}, X^{\DOWN c})$ is a polarity.
A \emph{polarity} $^\bot \colon P \to P$ on an ordered set $P$ is defined so that $x^{\bot \bot} = x$ and $x \leq y$ implies $x^\bot \geq y^\bot$ for $x,y \in P$.
Such a polarity is an order-isomorphism from $(P, \leq)$ to $(P,\geq)$. Hence $P$ is isomorphic to its dual.

It was proved in \cite{PomPom88} that $\mathit{RS}$  is a complete  sublattice of $\wp(U) \times\wp(U)$ ordered by the coordinatewise set-inclusion relation, which 
means that $\mathit{RS}$ is an algebraic completely distributive lattice such that
\begin{equation} \label{Eq:RS_meet}
\bigwedge \{  ( X^\DOWN, X^\UP ) \mid X\in\mathcal{H} \}  =
\big ( \bigcap_{X \in\mathcal{H}} X^\DOWN, \bigcap_{X\in\mathcal{H}} X^\UP \big )
\end{equation}
and
\begin{equation} \label{Eq:RS_join}
\bigvee \{  ( X^\DOWN, X^\UP ) \mid X\in\mathcal{H} \}  =
\big ( \bigcup_{X\in\mathcal{H}} X^\DOWN, \bigcup_{X\in\mathcal{H}} X^\UP \big )
\end{equation}
for any $\mathcal{H} \subseteq \wp(U)$. Note that an \emph{algebraic lattice} $L$ is
a complete lattice in which each element is a join of compact elements (defined on page~\pageref{Def:Compact})
and a \emph{completely distributive lattice} is a complete lattice in which arbitrary joins distribute over arbitrary meets.
They also proved that $\mathit{RS}$ is a Stone lattice such that $(X^\DOWN, X^\UP)^* = (X^{c \DOWN}, X^{c \DOWN})$
for $X \subseteq U$. This result was improved by S.~D.~Comer \cite{Comer} by showing that 
\[ \mathbb{RS} = (\mathit{RS}, \vee, \wedge, ^*, ^+, (\emptyset,\emptyset), (U,U))\]
is a regular double Stone algebra, where  $(X^\DOWN, X^\UP)^+ = (X^{c \UP}, X^{c \UP})$ for $X \subseteq U$. More importantly, he proved
that every regular double Stone algebra is isomorphic to a subalgebra of $\mathbb{RS}$ defined by some equivalence relation.

In the literature can be found numerous studies on rough sets that are determined 
by so-called \emph{information relations} reflecting distinguishability or indistinguishability 
of the elements of the universe of discourse; see \cite{Orlowska1998} and the references therein.
The idea is that $R$ may be an arbitrary binary relation, and rough lower and 
upper approximations are then defined in terms of $R$. Let us denote 
for any $x \in U$, $R(x) = \{ y \mid x \, R \, y\}$. For all $X \subseteq U$, the \emph{lower} and \emph{upper approximations} of $X$ are defined by
\[ X^\DOWN = \{ x \in U \mid R(x) \subseteq X \} \text{ \quad and \quad }  X^\UP = \{ x \in U \mid R(x) \cap X \ne \emptyset \}, \]
respectively. The \emph{rough equality} relation, the \emph{set of rough sets} $\textit{RS\/}$, and its partial order are defined as in case of equivalences. 

As shown in \cite{JRV09}, if $R$ is a quasiorder (a reflexive and transitive binary relation) on $U$, then 
$\mathit{RS}$ is a complete sublattice of $\wp(U) \times \wp(U)$ ordered by the componentwise inclusion, meaning that $\mathit{RS}$ 
is an algebraic completely  distributive lattice such that the lattice-operations are defined as in \eqref{Eq:RS_meet} and  \eqref{Eq:RS_join}
for all $\mathcal{H} \subseteq \mathit{RS}$. As in the case of equivalences, the map ${\sim} \colon (X^\DOWN, X^\UP) \mapsto (X^{c \DOWN}, X^{c \UP})$  
is a polarity on $\mathit{RS}$. In fact (see \cite{JarRad11,JPR13}), if $R$ is a quasiorder, then the algebra 
\[ \mathbb{RS} = (\mathit{RS},\vee,\wedge,\to,{\sim},0,1)\] 
forms a Nelson algebra with the operations:
\begin{align*}
 (X^\DOWN,X^\UP) \vee  (Y^\DOWN,Y^\UP)   & = (X^\DOWN \cup Y^\DOWN, X^\UP \cup Y^\UP), \\
 (X^\DOWN,X^\UP) \wedge (Y^\DOWN,Y^\UP) & = (X^\DOWN \cap Y^\DOWN, X^\UP \cap Y^\UP), \\
 (X^\DOWN,X^\UP)  \to  (Y^\DOWN,Y^\UP)  &= ( ( X^{\DOWN c} \cup Y^\DOWN)^\DOWN, X^{\DOWN c} \cup Y^\UP), \\
 {\sim}(X^\DOWN,X^\UP) &= ( X^{c \DOWN}, X^{c \UP}),\\ 
 0 &=  (\emptyset,\emptyset), \\
 1 &=  (U,U). 
\end{align*}
We proved in \cite{JarRad11} that if $\mathbb{A}$ is a Nelson algebra defined on an algebraic lattice, there exists a set $U$ and a quasiorder 
$R$ on $U$ such that $\mathbb{A}$ is isomorphic to the Nelson algebra $\mathbb{RS}$ determined by $R$. In \cite{JarRad14monteiro}, 
we generalized this representation theorem by stating that for any Nelson algebra $\mathbb{A}$, there exists a set $U$ and a 
quasiorder $R$ on $U$ such that $\mathbb{A}$ is isomorphic to a subalgebra of the Nelson algebra $\mathbb{RS}$ determined by $R$.

If $R$ is a tolerance (a reflexive and symmetric binary relation), then $\mathit{RS}$ is not necessarily a lattice \cite{Jarv99}.
A collection $\mathcal{H}$ of nonempty subsets of $U$ is called a \emph{covering} of $U$ if $\bigcup \mathcal{H} = U$.
A covering $\mathcal{H}$ is \emph{irredundant} if $\mathcal{H} \setminus \{X\}$ is not a covering of $U$ for any $X \in \mathcal{H}$.
Each covering $\mathcal{H}$ defines a tolerance $\bigcup \{ X \times X \mid X \in \mathcal{H}\}$, called the \emph{tolerance induced} 
by $\mathcal{H}$. In \cite{JarRad14}, we proved that if $R$ is a tolerance induced by an irredundant covering of $U$, then $\mathit{RS}$
is an algebraic and completely distributive lattice such that for any $\mathcal{H} \subseteq \wp(U)$,
\begin{equation} \label{Eq:RS_meet2}
\bigwedge \{  ( X^\DOWN, X^\UP ) \mid X \in\mathcal{H} \}  =
\Big ( \bigcap_{X \in\mathcal{H}} X^\DOWN, \big ( \bigcap_{X \in\mathcal{H}} X^\UP \big )^{\DOWN \UP} \Big )
\end{equation}
and
\begin{equation} \label{Eq:RS_join2}
\bigvee \{  ( X^\DOWN, X^\UP ) \mid X\in\mathcal{H} \}  =
\Big ( \big ( \bigcup_{X\in\mathcal{H}} X^\DOWN \big )^{\UP \DOWN} , \bigcup_{X\in\mathcal{H}} X^\UP \Big ).
\end{equation}
In addition, we showed in \cite{JarRad17}  that if $R$ is a tolerance induced by an irredundant covering, then
\[ (\mathit{RS},\vee,\wedge, {\sim}, ^*, 0, 1)\] 
is a pseudocomplemented Kleene algebra with the operations:
\begin{align*}
 (X^\DOWN,X^\UP) \vee  (Y^\DOWN,Y^\UP)   & = ( (X^\DOWN \cup Y^\DOWN)^{\UP \DOWN} , X^\UP \cup Y^\UP), \\
 (X^\DOWN,X^\UP) \wedge (Y^\DOWN,Y^\UP) & = (X^\DOWN \cap Y^\DOWN, (X^\UP \cap Y^\UP)^{\DOWN \UP} ), \\
 {\sim}(X^\DOWN,X^\UP) &= ( X^{c \DOWN}, X^{c \UP}),\\
 (X^{\DOWN}, X^{\UP})^* &=  (X^{c \DOWN \DOWN}, X^{c \DOWN \UP}),\\
 0 &=  (\emptyset,\emptyset), \\
 1 &=  (U,U). 
\end{align*}
We also proved in \cite{JarRad17} that if $\mathbb{L} = (L,\vee,\wedge,{\sim},{^*}, 0,1)$ is a regular pseudocomplemented Kleene
algebra defined on an algebraic lattice, then there exists a set $U$ and a tolerance $R$ induced by an irredundant
covering of $U$ such that $\mathbb{L}$ is isomorphic to the rough set pseudocomplemented Kleene algebra 
$\mathbb{RS} = (\mathit{RS},\vee,\wedge,{\sim}, {^*}, 0, 1)$  determined by $R$. 
Such an algebra can be considered as the algebraic counterpart of the three-valued Kleene logic. The role of  
this logic and its interrelation with incomplete formal contexts is discussed, for instance, in \cite{Burmeister2000}.
In this work we generalize our result by showing that if $\mathbb{L}$ is any regular pseudocomplemented Kleene algebra, 
then there exists a set $U$ and a tolerance $R$ induced by an irredundant covering of $U$ such that $\mathbb{L}$ is isomorphic 
to a subalgebra of the rough set pseudocomplemented Kleene algebra $\mathbb{RS}$ determined by $R$. Note that 
in \cite{Boffa2018} the authors show how certain sequences of approximation pairs determined by refining tolerances form 
finite centered Kleene algebras having the interpolation property. They also prove that every such a Kleene algebra is isomorphic 
to the algebra of sequences of approximation pairs of subsets of a suitable universe.

This work is structured as follows:  In the next section we recall some notions and facts related to De Morgan and Kleene algebras. 
In particular, we are interested in pseudocomplemented Kleene algebras and their regularity. In Section~\ref{Sec:VarletSpaces}, 
we introduce so called Kleene--Varlet spaces. These structures are essential tools for our main result of the section stating that
each regular pseudocomplemented Kleene algebra is isomorphic to a subalgebra of the rough set pseudocomplemented Kleene 
algebra defined by a tolerance. Section~\ref{Sec:StoneIdentity} is devoted to regular pseudocomplemented 
Kleene algebras satisfying the Stone identity $x^* \vee x^{**} = 1$. We show that there is a one-to-one correspondence between 
these algebras and regular double Stone algebras. The section ends by showing that each regular pseudocomplemented 
Kleene algebra satisfying the Stone identity is isomorphic to a subalgebra of a rough set Kleene algebra defined by an equivalence relation.
The last section considers Kleene--Varlet spaces of regular pseudocomplemented Kleene algebras satisfying the Stone identity.

\section{Regular pseudocomplemented Kleene algebras} \label{Sec:Preliminaries}

An algebra $(L, \vee, \wedge, {^*}, 0, 1)$ is a $p$-algebra if $(L, \vee, \wedge, 0, 1)$ is a bounded lattice and $^*$
is a unary operation on $L$ such that $x \wedge z = 0$ iff $z \leq x^*$. The element $x^*$ is the \emph{pseudocomplement} of $x$. 
A lattice $L$ in which each element has a pseudocomplement is called a \emph{pseudocomplemented lattice}.
It is well known that $x \leq y$ implies $x^* \geq y^*$. We also have for $x,y \in L$, 
\begin{align*}
x^* &= x^{***},\\
(x \vee y)^* &= x^* \wedge y^*, \\
(x \wedge y)^{**} &= x^{**} \wedge y^{**}.
\end{align*}

An algebra $(L, \vee, \wedge, {^*}, {^+}, 0, 1)$ is a double $p$-algebra if $(L, \vee, \wedge , ^*, 0, 1)$ is
a $p$-algebra and $(L, \vee, \wedge, {^+}, 0, 1)$ is a dual $p$-algebra  (i.e.\, $z \geq x^+$ iff $x \vee z = 1$ for all $x,y \in L$).
The element $x^+$ is the \emph{dual pseudocomplement} of $a$. If $x \leq y$, then $x^+ \geq y^+$.
In addition, 
\begin{align*}
x^+ &= x^{+++},\\
(x \wedge y)^+ &= x^+ \vee y^+,\\ 
(x \vee y)^{++} &= x^{++} \vee y^{++}. 
\end{align*}
Note that by definition $x \leq x^{**}$ and $x^{++} \leq x$. Therefore, in a double $p$-algebra $x^{++} \leq x^{**}$.

An algebra is called \emph{congruence-regular} if every congruence is determined by any class of it: two congruences are necessarily 
equal when they have a class in common. 
J.~Varlet has proved in  \cite{Varlet1972} that double $p$-algebras satisfying the condition
\begin{equation}
x^* = y^* \text{ and } x^+ = y^+ \text{ imply } x=y. \tag{M}
\end{equation}
are exactly the congruence-regular ones. It is also proved by him that for the distributive double $p$-algebras,
condition (M) is equivalent to condition
\begin{equation}
x \wedge x^+ \leq y \vee y^*. \tag{D}
\end{equation}
On the other hand, T.~Katri\v{n}\'{a}k \cite{Kat73} has shown that any 
congruence-regular double pseudocomplemented lattice is distributive. In what follows,
we use the shorter term ``regular'' instead of ``congruence-regular''.

A filter $F$ of a lattice $L$ is called \emph{proper}, if $F \ne L$. A proper filter $F$ is a \emph{prime filter} if $a \vee b \in F$ implies 
$a \in F$ or $b \in F$. The set of prime filters of $L$ is denoted by $\mathcal{F}_p(L)$, or by $\mathcal{F}_p$ if there is no danger of confusion. 
A filter $F$ is \emph{maximal} if $F$ is proper and there is no proper filter that is strictly greater than $F$. It can be shown by using
Zorn's Lemma that every proper filter can be extended to a maximal filter. It is also known that in distributive lattices,
each maximal filter is a prime filter, but the converse statement is not true in general.

Varlet \cite{Varlet1968,Varlet1972} has given for distributive double $p$-algebras the following characterization of regularity
in terms of prime filters.

\begin{proposition}\label{Prop:Varlet}
Let $(L, \vee, \wedge, {^*}, {^+}, 0, 1)$ be a distributive double $p$-algebra. The following are equivalent:
\begin{enumerate}[\normalfont (a)]
 \item $L$ is regular;
 \item Any chain of prime filters of $L$ has at most two elements.
\end{enumerate}
\end{proposition}

A \emph{De Morgan algebra} is an algebra $(L,\vee,\wedge,\sim,0,1)$ such that $(L,\vee,\wedge, 0, 1)$ is a bounded distributive
lattice and $\sim$ is a polarity on $L$, that is, it satisfies 
\begin{equation}
{\sim}{\sim} x =x \tag{DM1}
\end{equation}
and
\begin{equation}
 x \leq y \Longrightarrow {\sim} x \geq {\sim} y. \tag{DM2}
\end{equation}
Note that equationally the operation $\sim$ can be defined by
\[
x = {\sim}{\sim}x \qquad \text{and} \qquad {\sim}x \vee {\sim}y = {\sim}(x \wedge y).
\]
A \emph{Kleene algebra} is a De~Morgan algebra $(L,\vee,\wedge, {\sim}, 0, 1)$ satisfying
\begin{equation}
x \wedge {\sim} x \leq y \vee {\sim} y  \tag{K}
\end{equation}
In \cite[Lemma~1.1]{CigGal81} it is proved that in a Kleene algebra $(L,\vee,\wedge, {\sim}, 0, 1)$,  
\begin{equation}\label{Eq:PseudoKleene}
x \wedge y = 0 \text{ implies } y \leq {\sim} x.
\end{equation}

\begin{sloppypar}
A \emph{pseudocomplemented De Morgan algebra} $(L,\vee,\wedge,{\sim}, {^*}, 0,1)$ is such that
$(L,\vee,\wedge,{\sim}, 0,1)$ is a De~Morgan algebra and $(L,\vee,\wedge, {^*}, 0,1)$ is a $p$-algebra. 
In fact, any pseudocomplemented De Morgan algebra forms a double $p$-algebra, where the pseudocomplement operations
determine each other by
\end{sloppypar}
\begin{equation}\label{Eq:StarPlus}
{\sim} x^* = ({\sim} x)^+ \text{ \ and \ } {\sim} x^+ = ({\sim} x)^*.
\end{equation}
By \eqref{Eq:PseudoKleene} we have that in a pseudocomplemented Kleene algebra
\begin{equation}\label{Eq:Normality}
x^* \leq {\sim} x \leq x^+.
\end{equation}

H.~P.~Sankappanavar \cite{Sankappanavar86} has proved that a pseudocomplemented De Morgan algebra satisfying (M)
truly is a congruence-regular pseudocomplemented De Morgan algebra. Therefore, in the sequel we may call pseudocomplemented 
De Morgan and Kleene algebras \emph{regular} when they satisfy (M) or (D). Note that in a pseudocomplemented De Morgan algebra,
condition (M) is actually in the form
\[
x^* = y^* \text{ and } ({\sim} x)^* = ({\sim }y)^* \text{ imply } x=y. 
\]

Pseudocomplemented ortholattices are studied in \cite{Hahmann2009}. They are algebras 
$(L,\vee,\wedge, {^\bot}, {^*}, 0,1)$ such that $(L,\vee,\wedge, {^*}, 0,1)$ is a $p$-algebra and
$(L,\vee,\wedge,^{\bot}, 0,1)$ is an ortholattice, that is, $(L,\vee,\wedge, 0,1)$ is a bounded lattice and
${^\bot} \colon L \to L$ is a polarity satisfying $x \wedge x^\bot = 0$. This implies also $x \vee x^\bot = 1$, that is, 
$L$ is complemented by  ${^\bot}$. Also a pseudocomplemented ortholattice $(L,\vee,\wedge, {^\bot}, {^*}, 0,1)$ 
defines a double $p$-algebra by $x^+ = x^{\bot * \bot}$. In addition, a property corresponding to \eqref{Eq:Normality} holds, that is,
\[ x^* \leq  x^\bot \leq x^+.\]
Pseudocomplemented ortholattices are not generally distributive. In fact, a distributive ortholattice is a Boolean algebra.

An element $j$ of a complete lattice $L$ is called \emph{completely join-irreducible} if $j = \bigvee S$
implies $j \in S$ for every subset $S$ of $L$. Note that the least element $0$ of $L$ is not completely
join-irreducible. The set of completely join-irreducible elements of $L$ is denoted by 
$\mathcal{J}(L)$, or simply by $\mathcal{J}$ if there is no danger of confusion. 
In a distributive lattice $L$, the \emph{principal filter} ${\uparrow} j = \{x \in L \mid x \geq j\}$ 
of each $j \in \mathcal{J}$ is prime. 

A complete lattice $L$ is \emph{spatial} if for each $x \in L$,
\[ x = \bigvee \{ j \in \mathcal{J} \mid j \leq x \}. \]
An element $x$ of a complete lattice $L$ is said to be \emph{compact} if, for every $S \subseteq L$, 
\label{Def:Compact}
\[ x \leq \bigvee S \Longrightarrow x \leq \bigvee F \text{ for some finite subset $F$ of $S$}. \]
Let us denote by $\mathcal{K}(L)$ the set of compact elements of $L$.
A complete lattice $L$ is said to be \emph{algebraic} if for each $a \in L$,
\[ a  = \bigvee \{ x \in \mathcal{K}(L) \mid x \leq a\}.\]  

It is known (see e.g. \cite{JarRad17}) that every De~Morgan algebra defined on an algebraic lattice is spatial.
In case of regular pseudocomplemented Kleene algebras defined on algebraic lattices we presented in \cite{JarRad17}
the following variant of Proposition~\ref{Prop:Varlet}.

\begin{proposition} \label{Prop:Regular}
Let $(L,\vee,\wedge,{\sim}, {^*}, 0, 1)$ be a pseudocomplemented De Morgan algebra defined on an algebraic lattice. 
The following are equivalent:
\begin{enumerate}[\rm (a)]
 \item $L$ is regular.
 \item Any chain in $\mathcal{J}$ has at most two elements.
\end{enumerate}
\end{proposition}

\section{Alexandroff topologies and Kleene--Varlet spaces} \label{Sec:VarletSpaces}

First recall some facts of Alexandroff topologies from the literature \cite{Alex37,Birk37}.
An \emph{Alexandroff topology} is a topology that
contains also all arbitrary intersections of its members. Let $\mathcal{T}$ be an Alexandroff topology on $X$.
Then, for each $A \subseteq X$, there exists the \emph{smallest neighbourhood} (i.e.\, the smallest
open set containing $A$):
\[
N(A) = \bigcap \{ Y \in \mathcal{T} \mid A \subseteq Y \}.
\]
In particular, the smallest neighbourhood of a point $x \in X$ is denoted by $N(x)$. The family 
\begin{equation*}
\mathcal{J}(\mathcal{T}) = \{ N(x) \mid x \in X\} 
\end{equation*}
is the \emph{smallest base} of the Alexandroff topology $\mathcal{T}$. 
Each member $B$ of $\mathcal{T}$ can be expressed
as $B = \bigcup \{ N(x) \mid x \in B \}$. 

For an Alexandroff topology $\mathcal{T}$ on $X$, the ordered set $(\mathcal{T},\subseteq)$ is a complete lattice
in which 
\[ \bigvee  \mathcal{H} = \bigcup  \mathcal{H} \text{ \quad and \quad }  \bigwedge  \mathcal{H} = \bigcap  \mathcal{H} \]
for any $\mathcal{H} \subseteq \mathcal{T}$. This lattice $\mathcal{T}$ is spatial and $\mathcal{J}(\mathcal{T})$
is the set of completely join-ir\-re\-ducible elements. The complete Boolean lattice $\wp(X)$ is known to be algebraic
and completely distributive. Because $\mathcal{T}$ is a complete sublattice of $\wp(X)$, it is algebraic and
completely distributive.

Let $(X, \lesssim)$ be a quasiordered set. We may define an Alexandroff topology 
${\mathcal U}(X)$ on $X$ consisting of all upward-closed subsets of $X$ with respect to the
relation $\lesssim$. Formally,
\[
\mathcal{U}(X) = \{ A \subseteq X \mid (\forall x,y \in X) \; x \in A \ \ \& \ \ x \lesssim y \Longrightarrow y \in A \} \\
\]
We write ${\uparrow}x = \{y \in X \mid x \lesssim y\}$ also for the \emph{quasiorder filter} of $x$.
The set ${\uparrow}x$ is the smallest neighbourhood of the point $x$ in the Alexandroff topology $\mathcal{U}(X)$.

We may now define the rough approximation operators in terms of the quasiorder $\lesssim$ on $X$, that is, for any $A \subseteq U$,
\[ A^\DOWN = \{ x \in X \mid {\uparrow}x \subseteq A \} \text{ \quad and \quad } 
  A^\UP = \{ x \in X \mid {\uparrow}x \cap A \ne \emptyset \} \]
for any $A \subseteq X$. It is easy to see that
\[ \mathcal{U}(X) = \{  A^\DOWN \mid A \subseteq X\}, \]
which means that $A \mapsto A^\DOWN$ is the interior operator of the topology $\mathcal{U}(X)$.
The lattice $\mathcal{U}(X)$ is pseudocomplemented, in which
\[ A^* = A^{c\DOWN} = A^{\UP c} = \{ x \in X \mid {\uparrow} x \cap A = \emptyset\}. \]
This means that for any quasiordered set $(X,\lesssim)$, the algebra
\[ (\mathcal{U}(X), \cup, \cap, {^*}, \emptyset, X) \]
is a distributive $p$-algebra.

\medskip%

We studied in \cite{JarRad14monteiro} so-called Monteiro spaces in the setting of rough sets defined by quasiorders. 
Monteiro spaces were introduced by D.~Vakarelov in \cite{Vaka77}, where they were used for giving a representation theorem for Nelson algebras. 
Let us define two kinds of ``spaces''. Note that Kleene spaces were also defined by P.~Pagliani and M.~Chakraborty in \cite{pagliani2018}.

\begin{definition} \label{Def:Kleene-Varlet}
Let $(X,\leq,g)$ be a structure such that $(X,\leq)$ is a partially ordered set and $g$ is 
a map on $X$. If the map $g$ satisfies conditions
\begin{enumerate}[({J}1)]
 \item if $x \leq y$, then $g(x) \geq g(y)$,
 \item $g(g(x)) = x$,
 \item $x \leq g(x)$ or $g(x) \leq x$,
\end{enumerate}
then $(X,\leq,g)$ is called a \emph{Kleene space}. A Kleene space is a \emph{Kleene--Varlet space} if
\begin{enumerate}[({J}4)]
\item any chain in $(X,\leq)$ has at most two elements.
\end{enumerate}
\end{definition}

The idea is that Kleene--Varlet spaces will be used in a representation theorem of those pseudocomplemented Kleene
algebras that are regular, meaning that any chain of their prime filters has at most two elements.
This is abstracted in (J4). Note that Monteiro spaces are also Kleene spaces, but (J4) is replaced by so-called 
\emph{interpolation property}, which states that if $x, y \leq g(x), g(y)$ for some $x, y \in X$, then there 
exists $z \in X$ such that $x, y \leq z \leq g(x), g(y)$. In Monteiro spaces, the length of chains is not restricted.

\begin{remark} \label{Rem:Varlet}
Let $(X,\leq,g)$ be a Kleene--Varlet space. Here we present some observations related to the map $g$.
\begin{enumerate}[\rm (a)]
\item Conditions (J1) and (J2) mean that $g\colon X\rightarrow X$ is a polarity on $X$, that is, $(X,\leq)$ is isomorphic to its dual
$(X,\geq)$.

\item Condition (J3) means that any $x \in X$ is comparable with $g(x)$. Therefore $X$ can be divided into two disjoint parts in terms of $g$: 
\[ \{x \in X \mid x \leq g(x)\} \quad \text{ and } \quad  \{x \in X \mid x > g(x) \}.\]

\item Condition (J4) says that $X$ has at most two levels: $\{x \in X \mid x \leq g(x)\}$ is the ``lower level''
and $ \{x \in X \mid x > g(x) \}$ is the ``upper level''. 

\item If $g(x) = x$, then  $x$ is not comparable with any $y\neq x$. Indeed, if $x < y$, then $g(y) <g(x) = x < y$ is a chain with more 
than two elements, which contradicts (J4). Similarly, $y < x$ implies that $y < x = g(x) < g(y)$ is a chain of three elements, a contradiction again.
\end{enumerate}
\end{remark}

For a Kleene--Varlet space $(X,\leq,g)$, we define a map ${\sim} \colon \mathcal{U}(X) \to \mathcal{U}(X)$ by:
\[ {\sim} A = \{ x \in X \mid g(x) \notin A \} .\]
The operation $\sim$ is well defined. Indeed, let $A \in \mathcal{U}(X)$. If $x \in {\sim}A$ and $x \leq y$, 
then $g(x) \notin A$ and $g(x) \geq g(y)$. Because $A$ is upward-closed, we have $g(y) \notin A$ and $y \in {\sim} A$. 
Thus, ${\sim}A \in \mathcal{U}(X)$.

We can now write the following proposition.

\begin{proposition} \label{Prop:VarletRegular}
Let $(X,\leq, g)$ be a Kleene--Varlet space. Then, the algebra
\[ ( \mathcal{U}(X), \cup, \cap, {\sim}, {^*}, \emptyset, X) \]
is a regular pseudocomplemented Kleene algebra defined on an algebraic lattice.
\end{proposition}

\begin{proof} We already know that $(\mathcal{U}(X), \cup, \cap, {\sim}, {^*}, \emptyset, X)$
is a pseudocomplemented algebraic lattice. Next we verify that ${\sim}$ is a Kleene operation. 
This proof is modified from the one appearing in \cite{Vaka77}. Let $A,B \in  \mathcal{U}(X)$.

\medskip\noindent%
(DM1) $x \in A \iff g(g(x)) \in A \iff g(x) \notin {\sim}A \iff x \in {\sim} {\sim}A$. Thus $A = {\sim}{\sim}A$.

\medskip\noindent%
(DM2) Assume $A \subseteq B$. If $x \in {\sim}B$, then $g(x) \notin B$. This gives $g(x) \notin A$ and $x \in {\sim} A$.
 So, ${\sim} B \subseteq {\sim} A$. 
 
\medskip\noindent%
(K) Suppose that $A \cap {\sim} A \nsubseteq B \cup {\sim} B$. Then there exists $x \in X$ such that $x \in A \cap {\sim} A$ and 
$x \notin B \cup {\sim} B$. Therefore,
\[ x \in A, \quad g(x) \notin A, \quad x \notin B,  \quad g(x) \in B .\]
But from these we have
\[ x \nleq g(x) \qquad \text{ and } \qquad g(x) \nleq x,\]
a contradiction. We have now proved that $ ( \mathcal{U}(X), \cup, \cap, {\sim}, {^*}, \emptyset, X)$ is a pseudocomplemented Kleene algebra.

\medskip\noindent%
(M) As we have noted, the family $\mathcal{J}(\mathcal{U}(X)) = \{ {\uparrow} x \mid x \in X\}$ is the set of completely join-irreducible elements
of the complete lattice $\mathcal{U}(X)$. It is easy to observe that for all $x,y \in X$, $x \leq y$ if and only if ${\uparrow} y \subseteq {\uparrow} x$.
This implies that any chain in $\mathcal{J}(\mathcal{U}(X))$ has at most two elements, because $X$ satisfies this property. Therefore, by 
Proposition~\ref{Prop:Regular},
$(\mathcal{U}(X), \cup, \cap, {\sim}, {^*}, \emptyset, X)$ is a regular pseudocomplemented Kleene algebra.
\end{proof}

Let $(L, \vee, \wedge,  {\sim}, {^*}, 0, 1)$ be a regular pseudocomplemented Kleene algebra with $\mathcal{F}_p$ as the set of its prime filters.
The algebra $(L, \vee, \wedge, {^*}, {^+}, 0, 1)$ is a distributive double $p$-algebra, where $^+$ is defined as in \eqref{Eq:StarPlus}.
By Proposition~\ref{Prop:Varlet}, any chain in $(\mathcal{F}_p, \subseteq)$ has at most two elements.
We define for any $P \in \mathcal{F}_p$ the set 
\[ g(P) = \{ x \in L \mid {\sim} x \notin P \}. \]

\begin{lemma} \label{Lem:RegKleenePrime}
Let $(L, \vee, \wedge,  {\sim}, {^*}, 0, 1)$ be a regular pseudocomplemented Kleene algebra.
For any $P \in \mathcal{F}_p$, $g(P)$ is a prime filter.
\end{lemma}

\begin{proof}
Because $P$ is a proper filter, ${\sim}1 = 0 \notin P$. This means that $1 \in g(P)$ and therefore $g(P)$ is nonempty.

Assume that $x \in g(P)$ and $x \leq y$. Now ${\sim}x \notin P$ and ${\sim} x \geq {\sim}y$ imply ${\sim}y \notin P$, because $P$ is a filter.
Then $y \in g(P)$ and $g(P)$ is upward-closed.

Suppose $a,b \in g(P)$. Then, ${\sim}a \notin P$ and ${\sim}b \notin P$. Assume for contradiction that $a \wedge b \notin g(P)$.
Then ${\sim}(a \wedge b) = {\sim} a \vee {\sim} b$ belongs to $P$. But because $P$ is a prime filter, we have that ${\sim} a \in P$ or 
${\sim} b \in P$, a contradiction. Thus $a \wedge b \in g(P)$.

The filter $g(P)$ is proper, because $0 \in g(P)$ would imply ${\sim}0 = 1 \notin P$, which is impossible.

Finally, suppose $g(P)$ is not prime. Then, there are $a$ and $b$ in $L$ such that $a \vee b \in g(P)$, but $a \notin g(P)$ and $b \notin g(P)$.
Therefore, ${\sim} a \in P$ and ${\sim} b \in P$. Because $P$ is a filter then ${\sim}a \wedge {\sim} b = {\sim} (a \vee b)$ is in $P$. But 
this gives that $a \vee b \notin g(P)$, a contradiction. Hence $g(P)$ is prime. 
\end{proof}

Our next lemma shows how the prime filters of a regular pseudocomplemented Kleene algebra
form a Kleene--Varlet space.

\begin{lemma}\label{Lem:PseudoToVarlet}
If $(L, \vee, \wedge,  {\sim}, {^*}, 0, 1)$ is a regular pseudocomplemented Kleene algebra, then 
the triple $(\mathcal{F}_p, \subseteq, g)$ forms a Kleene--Varlet space.
\end{lemma}

\begin{proof} We will show that $g$ satisfies the conditions (J1)--(J4). Let $P$, $P_1$, and $P_2$ be prime filters of $L$.

\medskip\noindent%
(J1) Assume that $P_1 \subseteq P_2$. If $x \in g(P_2)$, then ${\sim}x \notin P_2$. This gives ${\sim}x \notin P_1$ and $x \in g(P_1)$.
Hence, $g(P_2) \subseteq g(P_1)$.

\medskip\noindent%
(J2) For any $x \in L$, $x \in P \iff {\sim}{\sim}x \in P \iff {\sim} x \notin g(P) \iff x \in g(g(P))$.

\medskip\noindent%
(J3) Suppose that $P \nsubseteq g(P)$ and $g(P) \nsubseteq P$. There there are elements $x,y \in L$ such that
\[ x \in P, \qquad x \notin g(P), \qquad y \in g(P), \qquad y \notin P.\]
These imply $x \in P$ and ${\sim} x \in P$. Thus, $x \wedge {\sim} x \in P$. But now 
 $x \wedge {\sim} x \leq y \vee {\sim} y$ give that $y \vee {\sim} y \in P$. Because $P$ is a prime filter,
have $y \in P$ or ${\sim} y \in P$. Because the latter is equivalent to $y \notin g(P)$, we have a contradiction.

\medskip\noindent%
(J4) This condition is clear by Proposition~\ref{Prop:Varlet}. 
\end{proof}

\begin{sloppypar}
By combining Proposition~\ref{Prop:VarletRegular} and Lemma~\ref{Lem:PseudoToVarlet}, we have that any regular pseudocomplemented Kleene algebra 
$(L, \vee, \wedge,  {\sim}, {^*}, 0, 1)$ determines a regular pseudocomplemented Kleene algebra
\end{sloppypar}
\[ ( \mathcal{U}(\mathcal{F}_p), \cup, \cap, {\sim}, {^*}, \emptyset, \mathcal{F}_p).\]
defined on an algebraic lattice. Recall that for all $A \in \mathcal{U}(\mathcal{F}_p)$:
\[ {\sim} A = \{ P \in \mathcal{F}_p \mid g(P) \notin A\} \text{ \ and \ }
 A^* = \{ P \in \mathcal{F}_p \mid {\uparrow}P \cap A = \emptyset\},          
\]
where ${\uparrow}P = \{Q \in \mathcal{F}_p \mid P \subseteq Q\}$.

\medskip\noindent%
For any element $x \in L$, we denote 
\[
h(x) = \{ P \in \mathcal{F}_p \mid x \in P\}.
\]
It is easy to see that $h(x) \in \mathcal{U}(\mathcal{F}_p)$. Namely, if $P \in h(x)$ and $P \subseteq Q$ for some $P,Q \in \mathcal{F}_p$, 
then $x \in P \subseteq Q$, that is, $Q \in h(x)$. Therefore, the mapping $h \colon L \to  \mathcal{U}(\mathcal{F}_p)$ is well defined.

\begin{proposition} \label{Prop:Embedding} 
The mapping $h$ is an embedding between pseudocomplemented Kleene algebras.
\end{proposition}

\begin{proof}
We first note that $h$ is an injection. Because $L$ is distributive, for any $x \neq y$ there exists a prime filter $P$ such that $x \in P$ and
$y \notin P$, or $x \notin P$ and $y \in P$. This means that $h(x) \neq h(y)$.

\medskip\noindent%
\underline{$h(0) = \emptyset$}: Prime filters must be proper filters. Therefore, $0$
does not belong to any prime filter.

\medskip\noindent%
\underline{$h(1) = \mathcal{F}_p$}: The greatest element $1$ must belong to all prime filters.

\medskip\noindent%
\underline{$h(x \vee y) = h(x) \cup h(y)$}: For any $P \in \mathcal{F}_p$,
$P \in h(x \vee y)  \iff x \vee y \in P \iff x \in P \text{ or } y \in P 
\iff P \in h(x) \text{ or } P \in h(y) \iff P \in h(x) \cup h(y)$.

\medskip\noindent%
\underline{$h(x \wedge y) = h(x) \cap h(y)$}: Let $P \in \mathcal{F}_p$. Then 
$P \in h(x \wedge y) \iff x \wedge y \in P \iff x \in P \text{ and } y \in P \iff
P \in h(x) \text{ and }  P \in h(y) \iff P \in h(x) \cap h(y)$.

\medskip\noindent%
\underline{$h({\sim}x) =  {\sim}h(x)$}: If $P \in \mathcal{F}_p$, then 
$P \in h({\sim}x) \iff {\sim} x \in P \iff x \notin g(P) \iff g(P) \notin h(x) \iff
P \in {\sim}h(x)$. 

\medskip\noindent%
\underline{$h(x^*) = h(x)^*$}: The structure of the proof is taken from the proof of Lemma 9.10.4 in \cite{orlowska2015dualities}: 

\medskip\noindent%
($\subseteq$) Suppose $P \in h(x^*)$, that is, $x^* \in P$. Let $Q \in {\uparrow} P$. Then $x^* \in Q$, which
gives $x \notin Q$, because otherwise $0 = x \wedge x^*$ in $Q$. This is not possible, since $Q$ is a prime filter and thus proper.
Then $Q \notin h(x)$ gives ${\uparrow}P \cap h(x) = \emptyset$ and $P \in h(x)^*$.

\medskip\noindent%
($\supseteq$) Assume $P \notin h(x^*)$, that is, $x^* \notin P$. Let $Q$ be a filter generated by $P \cup \{x\}$.
First we show that $Q$ is proper. Indeed, if $Q$ is not proper, then $0 \in Q$. Because $0 \notin P$ and $x \neq 0$,
we have that $0 = x \wedge y$ for some $y \in P$, because $Q$ is the filter generated by  $P \cup \{x\}$.
This implies $y \leq x^*$. Now $x^* \notin P$ gives $y \notin P$, a contradiction. Because $Q$ is a proper filter,
there exists a prime filter $W$ such that $P \subseteq Q \subseteq W$. 
Now $x \in P \cup \{x\} \subseteq Q \subseteq W$ gives $W \in h(x)$. 
Thus, ${\uparrow}P \cap h(x) \neq \emptyset$ and $P \notin h(x)^*$, as required.
\end{proof}

In \cite[Theorem 5.3]{JarRad17} we proved that any regular pseudocomplemented Kleene algebra defined on an algebraic lattice is isomorphic 
to a rough set Kleene algebra determined by a tolerance induced by an irredundant covering. 
If $(A, \vee, \wedge,  {\sim}, {^*}, 0, 1)$ is a regular pseudocomplemented Kleene algebra, then by Proposition~\ref{Prop:Embedding}
it is isomorphic to a subalgebra of $\mathbb{U} = (\mathcal{U}(X), \cup, \cap, {\sim}, {^*}, \emptyset, X)$.
Because $\mathbb{U}$ is a regular pseudocomplemented Kleene algebra defined on an algebraic lattice, there exists a tolerance induced 
by an irredundant covering such that its rough set regular pseudocomplemented Kleene algebra $\mathbb{RS}$
is isomorphic to $\mathbb{U}$. Therefore, we can write the following representation theorem.

\begin{theorem} \label{Thm:Main} 
Let $\mathbb{L}$ be a regular pseudocomplemented Kleene algebra. Then, there exists a set $U$ and
a tolerance $R$ induced by an irredundant covering of $U$ such that $\mathbb{L}$ is isomorphic to a subalgebra of $\mathbb{RS}$. 
\end{theorem}

\section{Regular pseudocomplemented Kleene algebras satisfying the Stone identity} \label{Sec:StoneIdentity}

A \emph{Stone algebra} is a distributive $p$-algebra $(L, \vee, \wedge, {^*}, 0, 1)$ satisfying the \emph{Stone identity}:
\begin{equation} \label{Eq:Stone}
x^* \vee x^{**} = 1.
\end{equation}
In a Stone algebra the identities 
\[
(x \wedge y)^* = x^* \vee y^* \qquad \text{and} \qquad (x \vee y)^{**} = x^{**} \vee y^{**}
\]
also hold. A \emph{double Stone algebra} is a distributive double $p$-algebra $(L, \vee, \wedge, {^*}, {^+}, 0, 1)$ satisfying 
\eqref{Eq:Stone} and 
\begin{equation} \label{Eq:dualStone}
x^+ \wedge x^{++} = 0.
\end{equation}
A double Stone algebra satisfies the identity $x^{*+} = x^{**}$, because
\[ x^{**} = x^{**} \wedge 1 = x^{**} \wedge (x^{*+} \vee x^*) = x^{**} \wedge x^{*+}, \]
and hence $x^{**} \leq x^{*+}$. The inequality $x^{*+} \leq x^{**}$ follows from $x^* \vee x^{**} = 1$. 
Similarly, we can show $x^{+*} = x^{++}$. Because $x^{++} \leq x^{**}$, we have 
\[ x^* = x^{***} \leq x^{++*} = x^{+++} = x^+.\]
A double Stone algebra is called \emph{regular} if it is regular as a double $p$-algebra, that is, it satisfies
(M) or (D). 

Varlet has proved in \cite{Varlet1968} that three-valued {\L}ukasiewicz algebras coincide with regular double Stone algebras.
Here we use similar technique to prove that regular double Stone algebras coincide with regular pseudocomplemented Kleene
algebras satisfying \eqref{Eq:Stone}. The proof of the following proposition
is modified from the proof of \cite[Theorem 4.4]{Boicescu91}.

\begin{proposition} \label{Prop:Correspondence}
Let  $(L, \vee, \wedge, {^*}, {^+}, 0, 1)$ be a regular double Stone algebra. If we define an operation $\sim$ by
\[ {\sim} x = (x \wedge x^+) \vee x^*,\]
then the algebra $(L, \vee, \wedge, {\sim}, {^*}, 0, 1)$ is a regular pseudocomplemented Kleene algebra satisfying \eqref{Eq:Stone}.
\end{proposition}

\begin{proof} By straightforward computation:
\begin{align*}
({\sim} x)^* &= (x \wedge x^+)^* \wedge x^{**} = (x^* \vee x^{+*}) \wedge x^{**} \\
             &=  (x^* \wedge x^{**}) \vee (x^{++} \wedge x^{**}) = x^{++}; \\
({\sim} x)^+ &= (x \wedge x^+)^+ \wedge x^{*+} = (x^+ \vee x^{++}) \wedge x^{**} = x^{**};\\
{\sim}\!{\sim}x &= ({\sim}x \wedge ({\sim}x)^+) \vee ({\sim}x)^* = (((x \wedge x^+) \vee x^*) \wedge x^{**}) \vee x^{++} \\
                &= (x \wedge x^+ \wedge x^{**}) \vee (x^* \wedge x^{**}) \vee x^{++} = (x \wedge x^+) \vee x^{++};\\
({\sim}\!{\sim}x)^* &= (x \wedge x^+)^* \wedge x^{++*} = (x^* \vee x^{+*}) \wedge x^+ = (x^* \vee x^{++}) \wedge x^+ \\
                    &= (x^* \wedge x^+) \vee (x^{++} \wedge x^+) = (x^* \wedge x^+) = x^*; \\
({\sim}\!{\sim}x)^+ &= (x \wedge x^+)^+ \wedge x^{+++} = (x^+ \vee x^{++}) \wedge x^+ = x^+.
\end{align*}
Because $({\sim}\!{\sim}x)^* = x^*$ and $({\sim}\!{\sim}x)^+ = x^+$, we obtain $x = {\sim}{\sim}x$ by (M). Now
\[
({\sim}x \vee {\sim y})^* = ({\sim}x)^* \wedge ({\sim}y)^* = x^{++} \wedge y^{++} = (x \wedge y)^{++} = ({\sim}(x \wedge y))^*
\]
and
\[
({\sim}x \vee {\sim y})^+ = ({\sim}x)^+ \wedge ({\sim}y)^+ = x^{**} \wedge y^{**} = (x \wedge y)^{**} = ({\sim}(x \wedge y))^+
\]
From this we have ${\sim}x \vee {\sim y} = {\sim}(x \wedge y)$. Therefore, $(L, \vee, \wedge, {\sim}, 0, 1)$ is a De~Morgan algebra.
Furthermore,
\[ x \wedge {\sim} x = x \wedge ( (x \wedge x^+) \vee x^*) = (x \wedge x^+) \vee (x \wedge x^*) = x \wedge x^+ \]
and
\[ y \vee {\sim} y = y \vee (y \wedge y^+) \vee y^* = y \vee y^*.\] 
We have 
\[ x \wedge {\sim} x = x \wedge x^+ \leq y \vee y^* = y \vee {\sim} y \]
by (D). So $(L, \vee, \wedge, {\sim}, ^*, 0, 1)$ is a pseudocomplemented Kleene algebra, which is regular and satisfies \eqref{Eq:Stone}
by assumption. 
\end{proof}

As we have noted, $x^* \leq x^+$ holds for any element $x$ of a double Stone algebra. This implies that
\[
(x \wedge x^+ ) \vee x^* = (x \vee x^*) \wedge (x^+ \vee x^*) = (x \vee x^*) \wedge x^+.
\]
Therefore, the operation $\sim$ may be defined also as ${\sim} x =  (x \vee x^*) \wedge x^+$ in
a regular double Stone algebra $(L, \vee, \wedge, {^*}, {^+}, 0, 1)$.

Let $(L,\vee,\wedge,{\sim}, {^*}, 0,1)$ be a pseudocomplemented Kleene algebra. 
Then $L$ is a double pseudocomplemented lattice in which the pseudocomplements $^*$ and $^+$ determine each other.
In particular, ${\sim} x^+ = ({\sim} x)^*$. Therefore, if \eqref{Eq:Stone} is satisfied in $L$, then
\begin{align*}
 1 &= ({\sim} x)^* \vee ({\sim} x)^{**} = {\sim}(x^+) \vee ({\sim}(x^+))^* \\
   &=  {\sim}(x^+) \vee  {\sim}(x^{++}) = {\sim}(x^+ \wedge x^{++}).
\end{align*}
This means that $x^+ \wedge x^{++} = 0$ and \eqref{Eq:dualStone} is valid in $L$. Therefore, we can write the following
proposition.

\begin{proposition}
If $(L,\vee,\wedge,{\sim}, {^*}, 0,1)$ is a pseudocomplemented Kleene algebra satisfying \eqref{Eq:Stone}, then 
$(L,\vee,\wedge,^*,^+,0,1)$ is a double Stone algebra which is regular exactly when 
the pseudocomplemented Kleene algebra $(L,\vee,\wedge,{\sim}, {^*}, 0,1)$ is regular.
\end{proposition}

We have now shown that each regular double Stone algebra 
\[ \mathbb{L} = (L, \vee, \wedge, {^*}, {^+}, 0, 1)\]
defines a regular pseudocomplemented Kleene algebra $\mathbb{L}^\textrm{rpK}$ satisfying the identity $x^* \vee x^{**} = 1$, and each
regular pseudocomplemented Kleene algebra 
\[ \mathbb{K} = (K, \vee, \wedge, {\sim}, {^*}, 0, 1)\] satisfying  $x^* \vee x^{**} = 1$
defines a regular double Stone algebra $\mathbb{K}^\textrm{rdS}$. Our next proposition shows that the correspondences
$\mathbb{L} \mapsto \mathbb{L}^\textrm{rpK}$ and $\mathbb{K} \mapsto \mathbb{K}^\textrm{rdS}$ are one-to-one and mutually inverse.

\begin{proposition}\label{Prop:OneToOne}
Let $\mathbb{L}$ be a regular double Stone algebra and $\mathbb{K}$ be a regular pseudocomplemented Kleene algebra satisfying $x^* \vee x^{**} = 1$.
Then the following equalities hold:
\begin{enumerate}[\rm (a)]
\item $\mathbb{L} = (\mathbb{L}^\textrm{rpK})^\textrm{rdS}$;
\item $\mathbb{K} = (\mathbb{K}^\textrm{rdS})^\textrm{rpK}$.
\end{enumerate}
\end{proposition}

\begin{proof}
Let us first note that the operations $\vee$, $\wedge$, $^*$, $0$, $1$ are immutable in these transformations, because they are in
the signature of the both algebras. 

(a) Assume  $\mathbb{L} = (L, \vee, \wedge, {^*}, {^+}, 0, 1)$ is a regular double Stone algebra. It defines a regular
pseudocomplemented Kleene algebra $\mathbb{L}^\textrm{rpK}$ in which the operation $\sim$ is defined by
\[ {\sim} x = (x \wedge x^+) \vee x^*.\]
In  $\mathbb{L}^\textrm{rpK}$, a dual pseudocomplement is defined in terms of this $\sim$ and $^*$ by $x^\oplus = {\sim}({\sim} x) ^*$.
Now
\[ x^\oplus = {\sim}({\sim} x)^* = {\sim} x^{++} = (x^{++} \wedge x^{+++}) \vee x^{+++} = x^{+++} = x^+.\]
This means that the algebras $\mathbb{L}$ and $(\mathbb{L}^\textrm{rpK})^\textrm{rdS}$ coincide.

\smallskip

(b) Let $\mathbb{K}$ be a regular pseudocomplemented Kleene algebra satisfying $x^* \vee x^{**} = 1$. The corresponding
regular double Stone algebra is $\mathbb{K}^\textrm{rdS}$ where the dual pseudocomplement is defined by 
$x^+ = {\sim} ({\sim}x)^*$. In $\mathbb{K}^\textrm{rdS}$, a Kleene negation is defined by
\[ \neg x = (x \wedge x^+) \vee x^* .\]
According to the proof of Proposition~\ref{Prop:Correspondence}, we have $(\neg x)^* = x^{++}$.
Because $\mathbb{K}^\textrm{rdS}$ is a double Stone algebra, $x^{++} = x^{+*}$. On the other hand, 
as a pseudocomplemented Kleene algebra $\mathbb{L}$ satisfies \eqref{Eq:Normality}, and therefore 
$$x^{+*} \leq {\sim} x^+ \leq x^{++}.$$
We have ${\sim} x^+ =  x^{++} = x^{+*}$. By definition,  ${\sim} x^+ = ({\sim}x)^*$. Hence,
\[ (\neg x)^* = x^{++} = {\sim} x^+ = ({\sim}x)^* . \]

Similarly, by the proof of Proposition~\ref{Prop:Correspondence}, $(\neg x)^+ = x^{**}$. Since
$\mathbb{K}^\textrm{rdS}$ is a double Stone algebra, we have $x^{*+} = x^{**}$. By  \eqref{Eq:Normality},
$x^{**} \leq {\sim} x^* \leq x^{*+}$, and therefore ${\sim} x^* = x^{*+} = x^{**}$.
Because ${\sim}x^* = ({\sim}x)^+$, we can write 
\[ ({\neg} x)^+ = x^{**} = ({\sim}x)^+.\]

We have now proved $(\neg x)^* = ({\sim}x)^*$ and $({\neg} x)^+ = ({\sim}x)^+$. Because $\mathbb{K}^\textrm{rdS}$
is a regular double Stone algebra, we have $\neg x = {\sim}x$. 
\end{proof}

Since there is a one-to-one correspondence between regular double Stone algebras and regular pseudocomplemented Kleene
algebras satisfying \eqref{Eq:Stone}, and the pseudocomplements and dual pseudocomplements in double Stone algebras are 
unique, we can write the following corollary.

\begin{corollary} \label{Cor:UniqueKleene}
In any regular pseudocomplemented Kleene algebra satisfying identity \eqref{Eq:Stone}, the operation $\sim$ is unique.
\end{corollary}

\begin{example}\label{Ex:RegularDoubleStone}
\begin{sloppypar}
In a regular pseudocomplemented Kleene algebra, the operation $\sim$ is not necessarily unique. 
Let us consider the regular pseudocomplemented Kleene algebra $(L,\vee,\wedge,{\sim}, ^*,0,1)$ depicted in Figure~\ref{Fig:3x3}(a).
There are two ways to define the Kleene operation. The first way is 
\end{sloppypar}
\[ {\sim} 0 = 1, {\sim} a = g, {\sim} b = f, {\sim} d = d, \]
and the second is
\[ {\sim} 0 = 1, {\sim} a = f, {\sim} b = g, {\sim} d = d. \]
This is possible since $L$ does not satisfy \eqref{Eq:Stone}:
\[ a^* \vee a^{**} = b \vee b^* = b \vee a = d \neq 1.\]

The distributive bounded lattice in Figure~\ref{Fig:3x3}(b) is a well-known double Stone algebra. 
The only way to define a Kleene operation in
$L$ is by
\[ {\sim} 0 = 1, {\sim} a = g, {\sim} b = f, {\sim} c = e, {\sim} d = d. \]

\begin{figure}
\centering
\includegraphics[width=90mm]{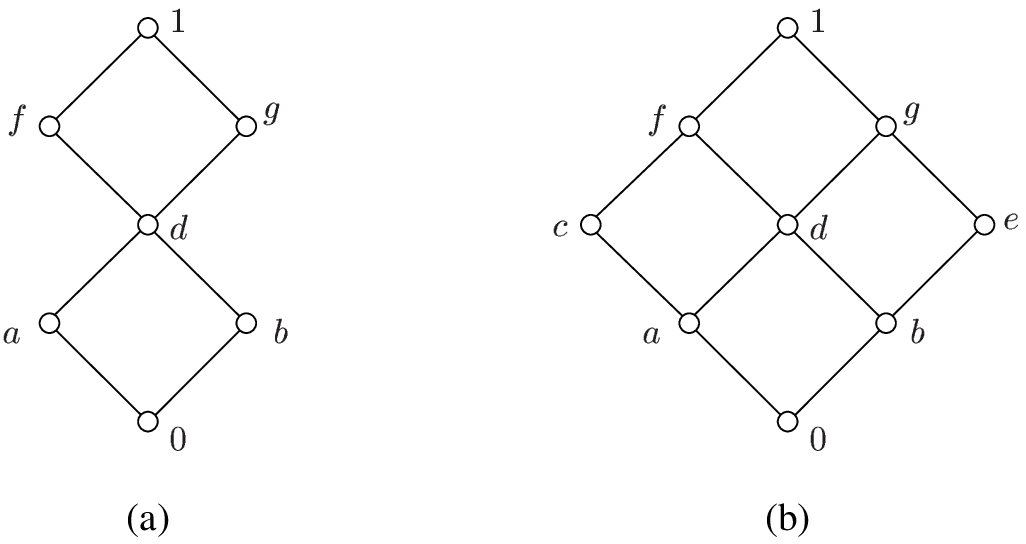}
\caption{\label{Fig:3x3}}
\end{figure}
\end{example}

We end this section by presenting a representation theorem for regular pseudocomplemented Kleene algebra
satisfying the Stone identity. Let $\mathbb{L} = (L, \vee, \wedge, {\sim}, {^*}, 0, 1)$ be a regular pseudocomplemented 
Kleene algebra satisfying $x^* \vee x^{**} = 1$. The unique regular double Stone
algebra corresponding $\mathbb{L}$ is $\mathbb{L}^\textrm{rdS}$. We may now apply the result by Comer mentioned in
Section~\ref{Sec:Intro}, which states that there exists a set $U$ and an equivalence $E$ on $U$ such that
$\mathbb{L}^\textrm{rdS}$ can be embedded to 
\[ \mathbb{RS} = (\mathit{RS},\vee,\wedge,^*,^+,0,1),\]
the rough set regular double Stone algebra defined by $E$. By the above, $\mathbb{RS}$ uniquely determines
a regular pseudocomplemented Kleene algebra $\mathbb{RS}^\textrm{rpK}$ satisfying $x^* \vee x^{**} = 1$. 
Obviously, the original regular pseudocomplemented Kleene algebra $\mathbb{L}$ can be embedded to $\mathbb{RS}^\textrm{rpK}$.
Therefore, we can write the following theorem. 

\begin{theorem} \label{Thm:MainB} 
Let $\mathbb{L}$ be a regular pseudocomplemented Kleene algebra satisfying $x^* \vee x^{**} = 1$.
Then, there exists a set $U$ and an equivalence $E$ on $U$ such that $\mathbb{L}$ is isomorphic to a subalgebra of 
\[ (\mathit{RS},\vee,\wedge,{\sim},^*,(\emptyset,\emptyset), (U,U)), \]
the pseudocomplemented Kleene algebra defined by $E$. 
\end{theorem}

\section{Kleene--Varlet spaces for regular pseudocomplemented Kleene algebras satisfying the Stone identity}

It is proved in \cite[Theorem 1]{Varlet1966} that a distributive pseudocomplemented lattice is a Stone lattice if and only if
every prime filter is contained in only one proper maximal filter. It is known that in a distributive lattice each maximal proper
filter is a (maximal) prime filter. This means that a distributive double $p$-algebra is a double Stone algebra if and only if 
each prime filter is included in a unique maximal prime filter and includes a unique minimal prime filter. 

If we combine this with the claim of Proposition~\ref{Prop:Varlet} stating that a distributive double $p$-algebra is
regular if and only if any chain of prime filters of $L$ has at most two elements, we have that 
a distributive double $p$-algebra is a regular double Stone algebra if and only if the family of its prime filters
is a disjoint union of chains of at most two elements. 
Notice that I.~D{\"u}ntsch and E.~Or{\l}owska considered in \cite{Duntsch2011}
so-called \emph{double Stone frames} $(X, \leq)$ which are partially ordered sets such that:
\begin{enumerate}[({F}1)]
 \item For every $x \in X$ there exists exactly one $y \in X$ such that $x \leq y$ and $y$ is maximal in $X$.
 \item For every $x \in X$ there exists exactly one $y \in X$ such that $x \geq y$ and $y$ is minimal in $X$.
\end{enumerate}
The second part of the proof of the following proposition is modified from the proof of
\cite[Theorem 4.5]{Duntsch2011}.

\begin{proposition} \label{Prop:DisjointVarlet}
Let $(X,\leq,g)$  be a Kleene--Varlet space. Then the  regular pseudocomplemented Klee\-ne algebra
$(\mathcal{U}(X), \cup, \cap, {\sim}, ^*, \emptyset, X)$ defined by $(X,\leq,g)$ satisfies
the Stone identity \eqref{Eq:Stone} if and only if $(X,\leq)$ is a union of disjoint chains of at most two elements.
\end{proposition}
 
\begin{proof} Suppose $(\mathcal{U}(X), \cup, \cap, {\sim}, ^*,\emptyset, X)$ satisfies the Stone identity \eqref{Eq:Stone}.
Because $(X,\leq,g)$ is a Kleene--Varlet space, each chain in $(X,\leq)$ has at most two elements. 
If $C = \{x\}$ is maximal chain of one element, then $x = g(x)$. By Remark~\ref{Rem:Varlet}(d),
$x$ is not comparable with any other element in $X$ and $C$ cannot have
common elements which other maximal chains. 

Next we show that if $C_1$ and $C_2$ are two-element chains
having a common element, then necessarily $C_1 = C_2$. Let us suppose that $C_{1} = \{x_1, y\}$ and $C_{2} = \{x_2,y\}$
such that either  (a) $x_1,x_2 < y$ or (b) $y < x_1,x_2$ in  Figure~\ref{Fig:overlap_chain} holds. 
Note that in (b), $g(x_{1}),g(x_{2})<g(y)$. Hence, it suffices to consider only (a).

\begin{figure}[h] 
\centering
\includegraphics[width=70mm]{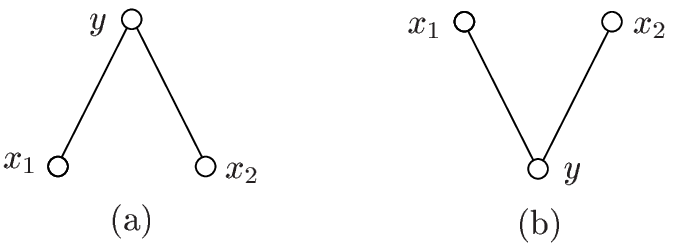}
\caption{\label{Fig:overlap_chain}}
\end{figure}

For case (a), we show that $g(x_{1}) = y$. 
Suppose for contradiction that $g(x_{1}) \neq y$.  As $x_{1} < y$, $x_1$ belongs to the ``lower level'' $\{x \in X \mid x < g(x)\}$ of $X$ and
${\uparrow} g(x_1) = \{ g(x_1) \}$ by Remark~\ref{Rem:Varlet}. 
By the definition of the pseudocomplement in $\mathcal{U}(X)$, 
\[ 
( {\uparrow} g(x_{1}) )^* = \{x \in X \mid {\uparrow} x \cap \{g(x_{1})\} = \emptyset\} = \{x \in X \mid x \nleq g(x_{1}) \}.
\] 
Because $x_1 < g(x_1)$, we get $x_{1} \notin ( {\uparrow} g(x_{1}) )^*$.  Observe that $y < g(x_{1})$ implies that $x_{1} < y < g(x_{1})$ 
is a chain of more than two elements. This is not possible, so $y \not < g(x_{1})$. Since $y \neq g(x_{1})$ holds by assumption, we obtain
$y \nleq g(x_{1})$ and $y \in ({\uparrow} g(x_{1}))^*$. As $x_1 < y$, we have ${\uparrow} x_{1} \cap ( {\uparrow} g(x_{1}))^* \neq \emptyset$. 
By definition,
\[ 
( {\uparrow} g(x_{1}) )^{**} = \{x \in X \mid {\uparrow} x \cap ({\uparrow} g(x_{1}) )^* = \emptyset \}.
\] 
Therefore, $x_{1} \notin ( {\uparrow} g(x_{1}) )^{* *}$. We already showed that $x_{1} \notin ( {\uparrow} g(x_{1}) )^*$.
Thus, 
\[ x_{1} \notin ( {\uparrow} g(x_{1}))^* \cup ({\uparrow} g(x_{1}))^{* *}.\] 
This contradicts $({\uparrow} g(x_{1}) )^* \cup ({\uparrow} g(x_{1})) ^{* *} = X$. Therefore, $g(x_{1}) \neq y$ is false
and so $g(x_{1})=y$. In an analogous way we can show that $g(x_{2})=y$.  Now $g(x_1) = y = g(x_2)$ implies $x_1 = x_2$ and $C_1 = C_2$. 
Thus, $(X,\leq)$ is a union of disjoint chains of at most two elements.

\smallskip%

Conversely, let $(X,\leq)$ be a union of disjoint chains of at most two elements and $A \in \mathcal{U}(X)$. 
Suppose that $x \notin A^* \cup A^{**}$. Because $x \notin A^*$, there is an element $y \in A$ such that
$x \leq y$. Similarly, $x \notin A^{**}$ means that there is  $z \in A^*$ such that $x \leq z$.
Because $(X,\leq)$ consists of disjoint chains of at most two elements, there is a chain $C$ such that
$x,y,z \in C$. Let $w$ be the biggest element in $C$. Then $y,z \leq w$ implies $w \in A \cap A^* = \emptyset$, a contradiction.
Hence, $A^* \cup A^{**} = X$.
\end{proof}

Let $(X,\leq,g)$ be a Kleene--Varlet space. For any $A\subseteq X$, let us denote 
\[ g[A] = \{g(a) \mid a \in A\}.\]
We can now write the following description of the pseudocomplement.

\begin{lemma} \label{Cor:StoneanCase}
Let $(X,\leq,g)$ be a Kleene--Varlet space such that $(X,\leq)$ is a union of disjoint chains of at most two elements.
For any $A \in \mathcal{U}(X)$,
\[ A^* = (A \cup g[A])^{c}.\]
\end{lemma}

\begin{proof}
If $x \in A^*$, then ${\uparrow}x \cap A = \emptyset$. This directly gives $x \notin A$. We have $x \leq g(x)$ or $g(x) \leq x$.
If $x \leq g(x)$, then ${\uparrow}x \cap A = \emptyset$ gives $g(x) \notin A$. If $g(x) \leq x$, then $g(x) \in A$ gives
$x \in A$, a contradiction. Thus, $g(x) \notin A$ and $x \notin g[A]$. We have shown $x \in (A \cup g[A])^{c}$.

Conversely, assume $x \notin A^*$. Then there is an element $y \in A$ such that $x \leq y$. Because $(X,\leq)$ consists
of disjoint chains of at most two elements,  $y = x$ or $y = g(x)$. This means $x \in A$ or $g(x) \in A$, that is,
$x \notin (A \cup g[A])^c$.
\end{proof}

We end this work showing the connections between regular pseudocomplemented Klee\-ne algebras, Kleene--Varlet spaces of their prime filters, and 
Kleene algebras of the Alexandroff topologies of upward-closed of prime filters.

\begin{theorem}\label{Thm:StoneanCase}
Let $\mathbb{L} = (L,\vee,\wedge,{\sim},^*,0,1)$ be a regular pseudocomplemented Kleene algebra. 
The following are equivalent:
\begin{enumerate}[\rm (a)]
\item The algebra $\mathbb{L}$ satisfies $x^* \vee x^{**} = 1$.
\item The Kleene--Varlet space $(\mathcal{F}_p,\subseteq,g)$ determined by $\mathbb{L}$ is a union of disjoint chains of at most two elements.
\item The regular pseudocomplemented Kleene algebra $(\mathcal{U}(\mathcal{F}_p), \cup, \cap, {\sim}, {^*}, \emptyset, \mathcal{F}_p)$
defined by $(\mathcal{F}_p,\subseteq,g)$ satisfies the Stone identity \eqref{Eq:Stone}.
\end{enumerate}
\end{theorem}

\begin{proof}
(a)$\Rightarrow$(b): By Lemma~\ref{Lem:PseudoToVarlet}, ($\mathcal{F}_p,\subseteq,g)$ is a Kleene--Varlet space, which means
that each chain has at most two elements. If $\mathbb{L}$ satisfies $x^* \vee x^{**} = 1$, then $\mathbb{L}^\textrm{rdS}$
is a double Stone algebra. As we noted in the beginning of this section, a distributive double $p$-algebra is a 
double Stone algebra if and only if each prime filter is included in a unique maximal prime filter and includes a unique minimal prime filter. 
By combining these two observations, we have that \mbox{($\mathcal{F}_p,\subseteq)$} is a union of disjoint chains of at most two elements.

(b)$\Rightarrow$(c): This follows directly from Proposition~\ref{Prop:DisjointVarlet}. 

(c)$\Rightarrow$(a): The mapping $h(x) = \{P \in \mathcal{F}_p \mid x \in P\}$ is a homomorphism by Proposition~\ref{Prop:Embedding}.
This means that
\[ h(x^* \vee x^{**}) = h(x^*) \vee h(x^{**}) =  h(x)^* \cup h(x)^{**} = \mathcal{F}_p = h(1).\]
Because $h$ is also an embedding, it is an injection. This gives $x^* \vee x^{**} = 1$. 
\end{proof}

\providecommand{\bysame}{\leavevmode\hbox to3em{\hrulefill}\thinspace}
\providecommand{\MR}{\relax\ifhmode\unskip\space\fi MR }
\providecommand{\MRhref}[2]{%
  \href{http://www.ams.org/mathscinet-getitem?mr=#1}{#2}
}
\providecommand{\href}[2]{#2}

\end{document}